\documentclass[12pt]{amsart}
\usepackage{amssymb,bbold}
\usepackage[all]{xy}

\theoremstyle{plain}
\newtheorem{theorem}{Theorem}
\newtheorem{corollary}[theorem]{Corollary}
\newtheorem{lemma}[theorem]{Lemma}
\newtheorem{proposition}[theorem]{Proposition}

\theoremstyle{definition}

\newtheorem{remark}[theorem]{Remark}

\newtheorem{question}{Question}

\begin{document}
\baselineskip 18pt

\title[Fremlin tensor product behaves well with the unbounded order convergence]
      {Fremlin tensor product behaves well with the unbounded order convergence}

\author[O.~Zabeti]{Omid Zabeti}
\address[O.~Zabeti]
  {Department of Mathematics, Faculty of Mathematics, Statistics, and Computer science,
   University of Sistan and Baluchestan, Zahedan,
   P.O. Box 98135-674. Iran}
\email{o.zabeti@gmail.com}
\keywords{Fremlin tensor product, unbounded order convergence, order convergence, vector lattice.}
\subjclass[2020]{Primary:  46M05. Secondary:  46A40.}
\maketitle

\begin{abstract}
Suppose $\Sigma$ is a topological space and $S(\Sigma)$ is the vector lattice of all equivalent classes of continuous real-valued functions defined on open dense subsets of $\Sigma$. In this paper, we establish some lattice and topological aspects of $S(\Sigma)$. In particular, as an application,  we show that the unbounded order convergence and the order convergence are stable under passing to the Fremlin tensor product of two Archimedean vector lattices.
\end{abstract}

\date{\today}

\maketitle
\section{Introduction}
 Suppose $E$ and $F$ are Archimedean vector lattices. Fremlin in \cite{Fremlin:72} has constructed the Archimedean vector lattice $E\overline{\otimes}F$ (known as the Fremlin tensor product of two Archimedean vector lattices). Under this structure, the algebraic tensor product $E\otimes F$ can be considered as an ordered vector subspace of $E\overline{\otimes}F$ with some valuable density properties (see \cite[1A]{Fremlin:74} for details). There are various properties of $E$ and $F$ that can be or can not be inherited by the Fremlin tensor product. For example, the Fremlin tensor product behaves well with sublattices; that is if $E_0$ and $F_0$ are vector sublattices of  $E$ and $F$, respectively, then $E_0\overline{\otimes}F_0$ is a vector sublattice of $E\overline{\otimes}F$. Furthermore, if $E_0$ and $F_0$ are order dense or regular, then, so is the Fremlin tensor product $E_0\overline{\otimes}F_0$ (for more details, see\cite[Proposition 3.2]{Gr:23}). However, there are many properties that fail in the Fremlin tensor product even for some classical spaces. For example, order completeness can not be preserved by the Fremlin tensor product (\cite[Example 3.8]{Gr:23} or \cite[Example 4C]{Fremlin:74}). Moreover, the Fremlin tensor product does not respect ideals or bands in general (see \cite{BT:23} for more details). Since tensor products are a useful tool in studying other mathematical subjects, it is interesting in its own right, to see whether or not some properties of the underlying spaces can be inherited by the tensor product. One of the possible questions arises while we are dealing with different convergence notions. For example, it is known that the projective tensor norm on the projective tensor product of two Banach spaces is a cross-norm so that it preserves the norm convergence. However, it does not respect the weak convergence; it is shown in \cite[Example 2.10]{Ryan} that the diagonal of the projective tensor product $\ell_2 \widehat{\otimes}\ell_2$ is isometrically isomorphic to $\ell_1$. Note that by the diagonal, we mean the closed subspace of $\ell_2\widehat{\otimes} \ell_2$ generated by the elementary tensors $e_n\otimes e_n$, in which $(e_n)_{n \in \Bbb N}$ is the standard basis of $\ell_2$. To see this, observe that if $u=\Sigma_{n=1}^{\infty}\alpha_n e_n\otimes e_n$, then it can be verified that the projective norm $u$, denoted by $\pi(u)$, is equal to $\Sigma_{n=1}^{\infty}|\alpha_n|$ so that $D$ is isometrically isomorphism to $\ell_1$. Now consider the standard basis $(e_n)\subseteq \ell_2$ which is weakly null but $e_n\otimes e_n$ is not since $\pi(e_n\otimes e_n)=1$ and in $\ell_1$, weakly convergence and norm convergence agree. It is known that, in the case of a vector lattice, we have two convergence structures: order convergence and unbounded order convergence. Therefore, it would be interesting to consider the following question:
\begin{question}
Does the Fremlin tensor product of two vector lattices respect either order convergence or unbounded order convergence?
\end{question}
Grobler in \cite[Corollary 3.4]{Gr:23} considered a "separate" version answer to this question. More precisely, he proved that for Archimedean vector lattices $E$ and $F$, if $\sigma:E\times F\to E\overline{\otimes}F$ is the natural lattice bimorphism ($(x,y)\rightarrow x\otimes y$), then $\sigma$ is order ($uo$-continuous) on each component. In this note, we show that the Fremlin tensor product of two Archimedean vector lattices behaves well with both order convergence and unbounded order convergence in a general sense; this extends the result of Grobler to a "jointly" version, as well. Before we reach to this goal, we need to verify some lattice and topological structures for the vector lattice consisting of equivalent classes of all continuous real-valued functions defined on open dense subsets of a topological space $\Sigma$.  
\section{preliminaries}
First, we recall some preliminaries regarding different notions of unbounded convergences. Let $E$ be a vector lattice. For a net $(x_{\alpha})$ in $E$, if there is a net $(u_\gamma)$, possibly over a
different index set, with $u_\gamma \downarrow 0$ and for every $\gamma$ there exists $\alpha_0$ such
that $|x_{\alpha} - x| \leq u_\gamma$ whenever $\alpha \geq \alpha_0$, we say that $(x_\alpha)$ converges to $x$ in order, in notation, $x_\alpha \xrightarrow{o}x$. A net $(x_{\alpha})$ in $E$ is said to be unbounded order convergent ($uo$-convergent) to $x \in E$ if for each $u \in E_{+}$, the net $(|x_{\alpha} - x| \wedge u)$ converges to zero in order (for short, $x_{\alpha}\xrightarrow{uo}x$). For order bounded nets, these notions agree together. For more details on these topics and related notions, see \cite{GTX:17}. 
For undefined terminology and general theory of vector lattices, we refer the reader to \cite{AB1, AB}.

In this part, we recall some notes about the Fremlin tensor product between vector lattices. For more details, see \cite{Fremlin:72, Fremlin:74}. Furthermore, for a comprehensive, new and interesting reference, see \cite{Wickstead1:24}. Furthermore, for a short and nicely written exposition on different types of tensor products between Archimedean vector lattices, see \cite{Gr:23}.

Suppose $E$ and $F$ are Archimedean vector lattices. In 1972, Fremlin constructed a tensor product $E\overline{\otimes} F$ that is an Archimedean vector lattice with the following features:
\begin{itemize}
\item {The algebraic tensor product $E\otimes F$ is a vector subspace of $E\overline{\otimes}F$. Therefore, we conclude that it is an ordered vector space in its own right}.
\item{The vector sublattice in $E\overline{\otimes} F$ generated by the algebraic tensor product $E\otimes F$ is $E\overline{\otimes}F$.}
\item{For each Archimedean vector lattice $G$ and every lattice bimorphism $\Phi:E\times F\to G$, there is a unique lattice homomorphism $T:E\overline{\otimes}F\to G$ such that $T(x\otimes y)=\Phi(x,y)$ for all $x\in E$ and for all $y\in F$. }
\end{itemize}
Therefore, we can consider every element of $E\overline{\otimes}F$  as a finite supremum and finite infimum of some elements of $E\otimes F$. 
Now, assume that $E$ is a vector lattice. It is shown in \cite{Wickstead:24} that $E$ has a representation as a vector sublattice of some $S(\Sigma)$-space, where $\Sigma$ is a topological space and $S(\Sigma)$ is the vector lattice of all (equivalence classes) of continuous real-valued functions defined on open dense subsets of $\Sigma$ (we shall speak about this space more in the next section). Now, we recall \cite[Proposition 3.1]{BW:17} that is crucial for the rest of the paper.

Assume that $E$ and $F$ are vector lattices with the representations as vector sublattices of some $S(\Sigma)$ and $S(\Omega)$ spaces, respectively, where $\Sigma$ and $\Omega$ are topological spaces. Then the Fremlin tensor product $E\overline{\otimes}F$ is isomorphic to the vector sublattice of $S(\Sigma\times \Omega)$ generated by the functions $(\sigma,\gamma)\rightarrow x(\sigma)y(\gamma)$, for each $x\in E$ and for each $y\in F$. 
\section{main results}
Suppose $X$ is a topological space and consider the space $C^{\infty}(X)$ consisting of all continuous extended  real-valued functions that are finite except on a nowhere dense set. In general, this space need not be a vector space (see \cite[Section 2]{BW:17}). However, when $X$ is extremally disconnected, it is known that $C^{\infty}(X)$ is a vector lattice under the pointwise vector and order operations. Suppose $K_1$ and $K_2$ are compact Hausdorff spaces. It is known that $C(K_1)\otimes C(K_2)$ is norm (order) dense in $C(K_1\times K_2)$. On the other hand, by the known Kakutani's theorem (\cite[Theorem 4.21]{AB}), every Archimedean vector lattice with an order unit can be considered as an norm (order) vector sublattice of some $C(K)$-space ($K$ compact and Hausdorff). So, we can transfer problems regarding tensor product of Archimedean vector lattices with order units in terms of $C(K)$-spaces. However, the issue is that Archimedean vector lattices with order units are scarce in the category of all vector lattices. On the other hand, by the Maeda-Ogasawara theorem (\cite[Theorem 7.29]{AB1}), every Archimedean vector lattice can be considered as an order dense vector sublattice of some $C^{\infty}(X)$-space for some compact Hausdorff extremally disconnected topological space $X$. Now, if we want to develop a theory for tensor products in terms of $C^{\infty}(X)$-spaces, the problem is that the Cartesian product of two extremally disconnected topological spaces $X$ and $Y$, is not extremally disconnected, in general. So, $C^{\infty}(X\times Y)$ may not a vector lattice (see \cite[Example 3.8]{Gr:23}). This problem was remarkably solved by Buskes and Wickstead in 2017 (\cite{BW:17}) by introducing a new larger vector lattice.

Suppose $\Sigma$ is a topological space. By $S(\Sigma)$, we mean the space of all equivalence classes of continuous functions defined on open dense subsets of $\Sigma$ under the equivalence relation $f\thicksim g$ if they coincide on the intersection of their domains. This space is introduced by Buskes and Wickstead in \cite{BW:17}. It is a vector lattice under the pointwise lattice and vector operations.
Observe that any $f\in S(\Sigma)$ is in fact a class of functions; more precisely, $[f]=\{g\in S(\Sigma), f\thicksim  g\}$. Now, define $\phi$ from $C^{\infty}(\Sigma)$ into $S(\Sigma)$ defined via $\Phi(f)=[f]$; note that, in this case, we restrict every $f\in C^{\infty}(\Sigma)$ to the open dense subset whose image is real. $\Phi$ is clearly injective so that we may identify $C^{\infty}(\Sigma)$ with its image under $\Phi$ to think  $C^{\infty}(\Sigma)$  as a subset of $S(\Sigma)$. In particular, $C(\Sigma)$ can be treated as a subset (in fact, a vector sublattice) of $S(\Sigma)$. 


 Now, assume that $\Sigma$ is also extremally disconnected. Then $\Phi$ is a vector lattice homomorphism by the definition of vector and lattice operations in $S(\Sigma)$. Furthermore, $\Phi$ is onto by using \cite[Theorem 7.25]{AB1}. So, we have the following observation. 
\begin{lemma}\label{1}
Suppose $\Sigma$ is an extremally disconnected topological space. Then there exists a lattice isomorphism between $S(\Sigma)$ and $C^{\infty}(\Sigma)$.
\end{lemma}
Considering Lemma \ref{1} with \cite[Theorem 7.27]{AB1}, we have the following.
\begin{corollary}\label{77}
Suppose $\Sigma$ is an extremally disconnected topological space. Then, $S(\Sigma)$ is a universally complete vector lattice.
\end{corollary}

Note that $C(\Sigma)$, the space of all real-valued continuous functions on topological space $\Sigma$, can be considered as a vector sublattice of $S(\Sigma)$. But, we have more if we consider completely regular topological spaces.
\begin{lemma}
Suppose $\Sigma$ is a completely regular topological space. Then, $C(\Sigma)$ is order dense in $S(\Sigma)$.
\end{lemma}
\begin{proof}
Suppose $0\neq g\in {S(\Sigma)}_{+}$. So, there exists some $t\in Dom(g)$ such that $g(t)=r>0$. Since $g$ is continuous, there exists a non-empty open set $U\subseteq \Sigma$ containing $t$ such that $g(s)> \frac{r}{2}$ for each $s\in U\cap Dom(g)$. Since $\Sigma$ is completely regular, there is $0\neq f\in {C(\Sigma)}_{+}$ such that $f\leq r\textbf{1}$ and $f\equiv 0$ outside of $U$. So, it is easily seen that $0<f\leq g$.
\end{proof}

Now, we characterize $uo$-convergence in $S(\Sigma)$. The proof is similar to the proof of \cite[Theorem 7.1]{B:23}; see also \cite[Theorem 3.2]{BTr:23}.
\begin{lemma}\label{444}
Suppose $\Sigma$  is a completely regular topological space. For a net $(f_{\alpha})\subseteq {S(\Sigma)}$, $f_{\alpha}\xrightarrow{uo}0$ if and only if for each non-empty open set $U\subseteq \Sigma$ and for each $\varepsilon>0$, there exist a non-empty open set $V\subseteq U$ and an index $\alpha_0$ such that $|f_{\alpha}(t)|\leq\varepsilon$ for each $t\in V\cap Dom(f_{\alpha})$.
\end{lemma}
\begin{proof}
Suppose $f_{\alpha}\xrightarrow{uo}0$.  Suppose $U\subseteq \Sigma$ is a non-empty open set and $\varepsilon>0$ is arbitrary. Since $\Sigma$ is completely regular, there exists a non-zero positive continuous function $g$ on $\Sigma$ with $g\leq \textbf{1}_{U}$. By \cite[Theorem 6.2]{B:23}, we can find non-zero positive function $h\in S(\Sigma)$ and an $\alpha_0$ such that $(|f_{\alpha}|-\varepsilon h)^{+}\perp g$ for each $\alpha\geq \alpha_0$. Therefore, $(|f_{\alpha}|-\varepsilon h)^{+}$ vanishes on non-empty open subset $V:=supp (g)\subseteq U$. Thus, for each $x\in V\cap Dom(f_{\alpha})$ and for each $\alpha\geq \alpha_0$, $|f_{\alpha}(x)|\leq \varepsilon$.  

For the other direction, take any non-zero positive $h\in S(\Sigma)$. There are non-empty open set $U\subseteq \Sigma$ and $\varepsilon>0$ with $h\geq \varepsilon \textbf{1}_{U}$. By the assumption, there are non-empty open set $V\subseteq U$ and an $\alpha_0$ such that $|f_{\alpha}(x)|\leq \varepsilon$ for each $x\in V\cap Dom(f_{\alpha})$ and for each $\alpha\geq \alpha_0$. Since $\Sigma$ is  completely regular,  there is a non-zero positive continuous function $g$ with $g\leq \varepsilon \textbf{1}_{U}$. Now, we can see that $(|f_{\alpha}|-h)^{+}$ vanishes on $V$ so that disjoint with $g$. Again, using \cite[Theorem 6.2]{B:23}, convinces us that $f_{\alpha}\xrightarrow{uo}0$. 
\end{proof}

Moreover, we have the following standard facts. We present the proof for the sake of completeness.
\begin{lemma}\label{2}
Suppose $\Sigma$ and $\Omega$ are completely regular topological spaces. Then, $S(\Sigma)\overline{\otimes}S(\Omega)$ is order dense in $S(\Sigma\times \Omega)$.
\end{lemma}
\begin{proof}
Suppose $h\in {S(\Sigma\times \Omega)}_{+}$. There exists $(t_0,s_0)\in \Sigma\times \Omega$ such that $h(t_0,s_0)=r>0$. Since $h$ is continuous, there exist non-empty open sets $U\subseteq \Sigma$ and $V \subseteq \Omega$ such that $h(t,s)> \frac{r}{2}$ for all $(t,s)\in U\times V$. Since $\Sigma$ and $\Omega$ are completely regular, we can find $0\neq f\in {S(\Sigma)}_{+}$ such that $f\leq \sqrt{\frac{r}{2}} {\mathbf{1}}_{S(\Sigma)}$ and $f$ vanishes outside of $U$. Also, there is $0\neq g\in {S(\Omega)}_{+}$ with $g\leq \sqrt{\frac{r}{2}} {\mathbf{1}}_{S(\Omega)}$ and $g=0$ outside of $V$. We show that $0<f\otimes g\leq h$. Since both $f$ and $g$ are non-zero, $f\otimes g>0$. Suppose $(t,s)\in Dom(f)\cap Dom(g)\cap Dom(h)$. First assume that $(t,s)\in U\times V$. Then, $(f\otimes g)(t,s)=f(t)g(s)\leq \frac{r}{2}< h(t,s)$. If $(t,s)\notin U\times V$, then either $f(t)=0$ or $g(s)=0$. Therefore, $(f\otimes g)(t,s)=f(t)g(s)=0\leq h(t,s)$. 
\end{proof}
The proof of the following result follows from \cite[Proposition 3.2]{Gr:23}. 
\begin{lemma}\label{3}
Suppose $\Sigma$ and $\Omega$ are topological spaces. Furthermore, assume that $E$ is an order dense vector sublattice of $S(\Sigma)$ and $F$ is an order dense vector sublattice of $S(\Omega)$. Then, $E\overline{\otimes}F$ is order dense in $S(\Sigma)\overline{\otimes}S(\Omega)$.
\end{lemma}
\begin{theorem}\label{101}
Suppose $E$ and $F$ are Archimedean vector lattices. Moreover, assume that $f_{\alpha}\xrightarrow{uo}f$ in $E$ and $g_{\beta}\xrightarrow{uo}g$ in $F$. Then, $f_{\alpha}\otimes g_{\beta}\rightarrow f\otimes g$ in the Fremlin tensor product $E\overline{\otimes}F$.
\end{theorem}
\begin{proof}
  By the known Maeda-Ogasawara theorem (\cite[Theorem 7.29]{AB1}) and using Lemma \ref{1}, There are two compact Hausdorff extremally disconnected topological spaces $\Sigma$ and $\Omega$ such that $E$ is an order dense vector sublattice of $S(\Sigma)$ and $F$ is an order dense vector sublattice of $S(\Omega)$. Thus, by \cite[Theorem 1.23]{AB1}, $E$ and $F$ are regular sublattices in $S(\Sigma)$ and $S(\Omega)$, respectively. So, by \cite[Theorem 3.2]{GTX:17}, $f_{\alpha}\xrightarrow{uo}f$ in $S(\Sigma)$ and $g_{\beta}\xrightarrow{uo}g$ in $S(\Omega)$. On the other hand, by \cite[Proposition 3.1]{BW:17}, $E\overline{\otimes}F$ can be considered as a vector sublattice of $S(\Sigma\times \Omega)$ generated by the mappings $(h\otimes k)(t,s)=h(t)k(s)$ for each $h\in S(\Sigma)$ and for each $k\in S(\Omega)$. Compatible with \cite[Remark 4.1]{BTr:23} and Lemma \ref{1}, it is enough to show that $f_{\alpha}\otimes g_{\beta}\xrightarrow{uo}f\otimes g$ in $S(\Sigma\times \Omega)$.
  
  Now, we use Lemma \ref{444}. 
 Suppose $W\subseteq \Sigma\times \Omega$ is a non-empty open set and $\varepsilon \in (0,1)$ is arbitrary. There are some non-empty open sets $U\subseteq \Sigma$ and $V\subseteq \Omega$ such that $U\times V\subseteq W$. We can find non-empty open sets $U_1\subseteq U$ and $V_1\subseteq V$ 
 such that $|f|$ is uniformly bounded (say, by $K_1$) on $U_1$ and $|g|$ is uniformly bounded on $V_1$ (namely, by $K_2$). 
 Choose open subsets $U_2\subseteq U_1$ and $V_2\subseteq V_1$ with $|(f_{\alpha}-f)(t)|<\frac{\varepsilon}{K_1}$ for each  $t\in U_2\cap Dom(f_{\alpha})\cap Dom(f)$ and $|(g_{\beta}-g)(s)|<\frac{\varepsilon}{K_2}$ for each $s\in V_2\cap Dom(g_{\beta})\cap Dom (g)$, provided that $\alpha$ and $\beta$ are sufficiently large. 

 
 Then for each $(t,s)\in (U_2\times V_2)\cap((Dom(f_{\alpha})\cap Dom(f))\times (Dom(g_{\beta})\cap Dom(g)))=(U_2\cap Dom(f_{\alpha})\cap Dom(f))\times (V_2\cap Dom(g_{\beta})\cap Dom(g)))$ and for sufficiently large $\alpha$ and $\beta$, we have

\[|(f_{\alpha}\otimes g_{\beta}-f\otimes g)(t,s)|\leq |(f_{\alpha}\otimes(g_{\beta}-g))(t,s)|+|((f_{\alpha}-f)\otimes g)(t,s)|\]

\[=|f_{\alpha}(t)||(g_{\beta}-g)(s)|+|(f_{\alpha}-f)(t)||g(s)|<(\frac{K_1}{K_2}+\frac{K_2}{K_1}+\frac{\varepsilon}{K_1 K_2})\varepsilon.\]
Now, an easy application of Lemma \ref{2} and Lemma \ref{3} convinces us that $f_{\alpha}\otimes g_{\beta}\xrightarrow{uo}f\otimes g$ in $E\overline{\otimes}F$.


\end{proof}
\begin{proposition}\label{302}
Suppose $E$ and $F$ are vector lattices. Assume that $(x_{\alpha})\subseteq E$ is $uo$-null and $(y_{\beta})\subseteq F$ is eventually order bounded. Then, $x_{\alpha}\otimes y_{\beta}\xrightarrow{uo}0$ in $E\overline{\otimes}F$.
\end{proposition}
\begin{proof}
There exists $y\in F_{+}$ such that $|y_{\beta}|\leq y$ for sufficiently large $\beta$. Suppose $w\in (E\overline{\otimes}F)_{+}$. By  \cite[1A (d)]{Fremlin:74}), there exist $x_0\in E_{+}$ and $y_0\in F_{+}$ with $w\leq x_0\otimes y_0$. Now, we use the following inequality.
\[|x_{\alpha}\otimes y_{\beta}|\wedge w\leq (|x_{\alpha}|\otimes |y_{\beta}|)\wedge (x_0\otimes y_0)\]
\[\leq (|x_{\alpha}|\wedge x_0)\otimes (|y_{\beta}|\vee y_0)\leq (|x_{\alpha}|\wedge x_0)\otimes (y\vee y_0).\]
Now, by \cite[Corollary 3.4]{Gr:23}, we conclude that $x_{\alpha}\otimes y_{\beta}\xrightarrow{uo}0$, as claimed.
\end{proof}
\begin{remark}
Note that order bounded assumption is essential in Proposition \ref{302} and can not be removed. It is know that in $L_p(\mu)$-spaces, for sequences, $uo$-convergence and almost everywhere convergence agree. Put $E=L^2(\Bbb R)$, $f_n(x)=\frac{x}{n}$ and $g_n(x)=n$. It is easy to see that $(f_n)$ is $uo$-null and $(g_n)$ is not order bounded. Moreover, $(f_n\otimes g_n)_{n}$ is the identity mapping which is not pointwise convergent to zero, certainly. 
\end{remark}
It can be easily verified that if a net $(x_{\alpha})\subseteq E$ and a net $(y_{\beta})\subseteq F$ are eventually order bounded, then, the net $(x_{\alpha}\otimes y_{\beta})$ is also eventually order bounded in $E\overline{\otimes}F$. So, we have the following fact that is an extension of \cite[Corollary 3.4]{Gr:23}, as well.
\begin{corollary}
Suppose $E$ and $F$ are vector lattices. Assume that $(x_{\alpha})\subseteq E$ is order null and $(y_{\beta})\subseteq F$ is also order null. Then, $x_{\alpha}\otimes y_{\beta}\xrightarrow{o}0$ in $E\overline{\otimes}F$.
\end{corollary}

{\bf Acknowledgments}. Most of the work on this paper was done during a visit of the author to University of Alberta in 2023. The author would like to thank Professor Vladimir Troitsky so much for great hospitality during his visit, as well. Thanks is also due to Eugene Bilokopytov for many useful discussions.

\end{document}